\documentclass[12pt,reqno]{amsart}

\usepackage{amsfonts, amsthm, amsmath}

\usepackage{tikz}
\usepackage{tikz-cd}

\usepackage{amscd}

\usepackage[latin2]{inputenc}

\usepackage{t1enc}

\usepackage[mathscr]{eucal}

\usepackage{indentfirst}

\usepackage{graphicx}

\usepackage{graphics}

\usepackage{pict2e}

\usepackage{mathrsfs}

\usepackage{mathabx} 

\usepackage{enumerate}

\usepackage{booktabs} 
\usepackage[pagebackref]{hyperref}
\hypersetup{colorlinks=true}
\usepackage{cite}
\usepackage{color}
\usepackage{epic}
\usepackage{hyperref} 
\numberwithin{equation}{section}
\def\blue{\textcolor{blue}}
\def\red{\textcolor{red}}
\theoremstyle{plain}

\newtheorem{theorem}{Theorem}[section]

\newtheorem{lemma}[theorem]{Lemma}

\newtheorem{corollary}[theorem]{Corollary}

\newtheorem{proposition}[theorem]{Proposition}

\theoremstyle{definition}

\newtheorem{Def}[theorem]{Definition}

\newtheorem{example}[theorem]{Example}

\newtheorem{conjecture}[theorem]{Conjecture}

\newtheorem{?}[theorem]{Problem}

\usepackage{color}
\def\blue{\textcolor{blue}}
\def\red{\textcolor{red}}
\def\green{\textcolor{green}}

\newcommand\catalannumber[3]{
  \fill[gray!25]  (#1) rectangle +(#2,#2);
  \draw[help lines] (#1) grid +(#2,#2);
  \draw[dashed] (#1) -- +(#2,#2);
  \coordinate (prev) at (#1);
  \foreach \dir in {#3}{
    \ifnum\dir=0
    \coordinate (dep) at (0,1);
    \else
    \coordinate (dep) at (1,0);
    \fi
    \draw[line width=2pt,-stealth] (prev) -- ++(dep) coordinate (prev);
  };
}

\def\boxit#1{\leavevmode\hbox{\vrule\vtop{\vbox{\kern.33333pt\hrule
    \kern1pt\hbox{\kern1pt\vbox{#1}\kern1pt}}\kern1pt\hrule}\vrule}}

\newcommand{\Z}{\mathbb{Z}}

\def\Sym{\mathfrak{S}}
\def\FV{\mathrm{FV}}
\def\FZ{\mathrm{FZ}}

\def\CSZ{\mathrm{CSZ}}
\def\SZ{\mathrm{SZ}}

\def\rasc{\mathrm{2\underline{13}}}
\def\rdes{\mathrm{2\underline{31}}}
\def\ldes{\mathrm{\underline{31}2}}
\def\lasc{\mathrm{\underline{13}2}}

\def\nest{\mathrm{nest}}
\def\cros{\mathrm{cros}}
\def\ucr{\mathrm{ucr}}
\def\ecr{\mathrm{ecr}}
\def\lcr{\mathrm{lcr}}
\def\une{\mathrm{une}}
\def\ene{\mathrm{ene}}
\def\lne{\mathrm{lne}}
\def\wene{\widetilde{\mathrm{ene}}}
\def\wne{\widetilde{\mathrm{nest}}}
\def\Ene{\mathrm{Ene}}
\def\Wene{\widetilde{\mathrm{Ene}}}
\def\Aba{\mathrm{Aba}}
\def\Dtb{\mathrm{Dtb}}

\def\Ene{\mathrm{Ene}}

\def\baxter{\mathrm{3\underline{14}2,2\underline{41}3}}
\def\dualbax{\mathrm{\overline{2}41\overline{3},\overline{3}14\overline{2}}}

\def\rU{\operatorname{U}}
\def\rD{\operatorname{D}}
\def\rLr{\operatorname{L_r}}
\def\rLb{\operatorname{L_b}}
\def\ULr{\operatorname{UL_r}}
\def\DLb{\operatorname{DL_b}}

\def\LH{\mathfrak{L}}
\def\cs{\operatorname{cs}}
\def\Bax{\operatorname{Bax}}

\newcommand\varspace[1][.7em]{%
  \makebox[#1]{%
    \kern.07em
    \vrule height.3ex
    \hrulefill
    \vrule height.3ex
    \kern.07em
  }
}

\newcommand{\be}{\begin{equation}}
\newcommand{\ee}{\end{equation}}

\title[]{An involution for trivariate symmetries of vincular patterns}

\author[J.~N.~Chen]{Joanna N. Chen}
\address{College of Science, Tianjin University of Technology, Tianjin 300384, P.R. China}
\email{joannachen@tjut.edu.cn}

\author[S.~Fu]{Shishuo Fu}
\address{College of Mathematics and Statistics, Chongqing University \& Key Laboratory of Nonlinear Analysis and its Applications (Chongqing University), Ministry of Education, Chongqing 401331, P.R. China}
\email{fsshuo@cqu.edu.cn}

\author[J.~Zeng]{Jiang Zeng}
\address{Universite Claude Bernard Lyon 1, CNRS UMR 5208, Institut Camille Jordan, F-69622 Villeurbanne cedex, France}
\email{zeng@math.univ-lyon1.fr}



\date{\today}

\begin{document}


\begin{abstract}
We provide a bijective proof of the equidistribution of two pairs of vincular patterns in permutations, thereby resolving a recent open problem of Bitonti, Deb, and Sokal (arXiv:2412.10214). Since the bijection is involutive, we also confirm their conjecture on the equidistribution of triple vincular patterns. Somewhat unexpectedly, we show that this involution is closed on the set of Baxter permutations, thereby implying another trivariate symmetry of vincular patterns. The proof of this second result requires a variant of a characterization of Baxter permutations in terms of restricted Laguerre histories, first given by Viennot using the Fran\c con-Viennot bijection.
\end{abstract}

\maketitle


\section{Introduction}
A  permutation {\it pattern} is a sub-permutation of a longer permutation.
An occurrence of a pattern $p$ in a permutation $\sigma$ is a subsequence of $\sigma$ that is order-isomorphic to $p$. 
 A {\it vincular pattern} is a permutation containing underlined letters  indicating which adjacent pairs of entries need to occur consecutively. These patterns were introduced by Babson and Steingr\'{i}msson~\cite{BS00}, who showed that almost all known \emph{Mahonian} statistics could be expressed in terms of vincular patterns (i.e., the number of the occurrences of a certain vincular pattern is treated as a permutation statistic). For instance, $534261$ contains only one occurrence of the vincular pattern $\underline{43}21$ in its subsequence $5321$, while the subsequence $5421$ forms a pattern $4321$ but not the vincular pattern $\underline{43}21$. Given a (vincular) pattern $p$ and a permutation $\sigma$, we denote by $p(\sigma)$ the number of occurrences of the pattern $p$ in $\sigma$.
 More works related to vincular patterns can be found in the book exposition by Kitaev \cite[Chapter 7]{Kit11}.

 For a permutation $\sigma=\sigma_1\ldots \sigma_n$ of $[n]:=\{1, \ldots, n\}$ and a letter $1\le \ell\le n$, we define the following coordinate statistics:
\begin{align*}
&\rasc(\ell,\sigma) :=  \# \{ j : \sigma^{-1}_\ell <j <n  ~\text{and}~  \sigma_j < \ell < \sigma_{j+1}  \},  \\[3pt]
&\rdes(\ell,\sigma) :=  \# \{ j : \sigma^{-1}_\ell<j <n  ~\text{and}~  \sigma_{j+1} <\ell < \sigma_j  \},  \\[3pt]
&\ldes(\ell,\sigma) :=  \# \{ j : 1  <j <\sigma^{-1}_\ell  ~\text{and}~  \sigma_{j} < \ell < \sigma_{j-1} \}, \\[3pt]
&\lasc(\ell,\sigma) :=  \# \{ j : 1  < j <\sigma^{-1}_\ell  ~\text{and}~  \sigma_{j-1} < \ell < \sigma_{j} \},
\end{align*}
and the respective vincular  statistics
$
\rasc(\sigma),\; \rdes(\sigma),\; \ldes(\sigma)$, and $ \lasc(\sigma)
$
as the sum of their corresponding coordinate statistics over $1\leq \ell\leq n$. 

Claesson~\cite[Proposition 7]{Cl01}   proved that the four vincular patterns
$\rasc, \rdes, \lasc, \ldes$ are equidistributed on $\Sym_n$, the set of all permutations of $[n]$. Shin and Zeng~\cite[Eq.~(39)]{SZ12} (see also \cite[Lemma 3.1]{FTHZ19})  proved  the equidistribution of  the bi-statistics  $(\rasc, \ldes)$ and $(\rdes, \ldes)$  on $\Sym_n$  from the common continued fraction expansion of their generating functions, viz, 
\begin{align}\label{continued fraction}
\sum_{n\geq 0} \sum_{\sigma\in \Sym_n}p^{\rdes(\sigma)}q^{\ldes(\sigma)}\,x^n
=\cfrac{1}{1-\cfrac{[1]_{p,q}\, x}{1-\cfrac{[1]_{p,q}\, x}{1-\cfrac{[2]_{p,q}\, x}{1-\cfrac{[2]_{p,q} \,x}{1-\ddots}}}}}
=\sum_{n\geq 0} \sum_{\sigma\in \Sym_n}p^{\rasc(\sigma)}q^{\ldes(\sigma)}\,x^n,
\end{align}
where $[i]_{p,q}=\frac{p^i-q^i}{p-q}$  for $i\geq 1$.

Recently, Bitonti, Deb, and Sokal~\cite{BDS24} rederived the second equation in \eqref{continued fraction} and  considered the eight possible ordered pairs formed by 
taking one vincular pattern of the form $2\underline{ab}$ and one of the form $\underline{ab}2$:
\begin{figure}[h]
\begin{minipage}{.4\textwidth}
\begin{enumerate}
\item $(\rasc, \ldes)$
\item $(\ldes, \rasc)$
\item $(\rdes, \lasc)$
\item $(\lasc, \rdes)$
\end{enumerate}
\end{minipage}
\begin{minipage}{.4\textwidth}
 \[\begin{tikzcd}
(1)\arrow[leftrightarrow]{d}[swap]{c}\arrow[leftrightarrow]{r}{r} & (2)
 \arrow[leftrightarrow]{d}{c} \\
(3) \arrow[leftrightarrow]{r}{r} & (4)
\end{tikzcd}
\]
\end{minipage}

\begin{minipage}{.4\textwidth}
\begin{itemize}
\item[(5)] $(\rasc, \lasc)$
\item[(6)] $(\ldes, \rdes)$
\item[(7)] $(\rdes, \ldes)$
\item[(8)] $(\lasc, \rasc)$
\end{itemize}
\end{minipage}
\begin{minipage}{.4\textwidth}
 \[
\begin{tikzcd}
(5)\arrow[leftrightarrow]{r}{r} 
 \arrow[leftrightarrow]{d}[swap]{c}& (6) \arrow[leftrightarrow]{d}{c} \\
(7) \arrow[leftrightarrow]{r}{r} & (8)
\end{tikzcd}
\]
\end{minipage}
\end{figure}

They observed that  the first four of these are equidistributed, and also the last four by using complementation $\sigma\mapsto \sigma^c$ (that is, mapping \emph{letters}
 $\ell\mapsto n+1-\ell$) and using reversal $\sigma\mapsto \sigma^r$ (that is, mapping \emph{indices} $i\mapsto  n+1-i$).  We can illustrate 
 these equidistributions   by drawing a digraph on $\{(1), \ldots, (8)\}$ such that 
 there is an edge  between (i) and (j)  if and only if they are related via a transformation $r$ or $c$. Hence, combining with 
 \eqref{continued fraction},  which shows the equidistribution of  the bi-statistics $(\rasc,\ldes)$ and $(\rdes,\ldes)$, the eight bi-statistics (1)--(8) are equidistributed
on $\Sym_n$. The following problem~\cite[Open Problem~7.6]{BDS24} is quite natural, see also 
\cite[Remark on p.~1697]{SZ10} and \cite[Remark~4.6]{CF23} for similar problems.

\begin{?}\label{problem.BDS}
Find a direct bijective proof of the equidistribution 
$(\rasc, \ldes)\sim (\rdes, \ldes)$.
\end{?}
 
Bitonti et al.~\cite{BDS24} further considered the joint distribution of all four vincular patterns
by defining  the polynomials in four variables
\begin{equation}
P_n(p,q,r,s)=\sum_{\sigma\in \Sym_n}p^{\lasc(\sigma)}q^{\ldes(\sigma)}
r^{\rasc(\sigma)}s^{\rdes(\sigma)}.
\end{equation}
Combining the complementation symmetry and reversal symmetry on pairs of bi-statistics, they obtain the $\Z_2\times \Z_2$ symmetry of the polynomial $P_n$:
\begin{equation}\label{bds-symmetries}
P_n(p,q,r,s)=P_n(q,p, s,r)=P_n(s,r, q,p)=P_n(r,s, p,q).
\end{equation}
They further  verified that for $n\geq 5$ these are the only permutations of the four variables that leave the polynomial $P_n$ invariant. By setting one of the variables to 1, based on the empirical data for $n\leq 11$, they  made the following remarkable conjecture on trivariate symmetries of vincular patterns~\cite[Conjecture 7.7]{BDS24}.

\begin{conjecture}\label{conj:tri-variate}
We have the relations
\begin{align}
P_n(1,q,r,s)&=P_n(1, q, s,r),\label{sym1}\\
P_n(p,1,r,s)&=P_n(p,1, s,r),\nonumber \\
P_n(p,q,1,s)&=P_n(q,p, 1,s),\nonumber \\
P_n(p,q,r,1)&=P_n(q,p, r,1).\nonumber
\end{align}
\end{conjecture}
By the symmetries \eqref{bds-symmetries} the four conjectured relations are equivalent. Therefore it is enough to establish any one of them -- say \eqref{sym1}. 
 A natural source of inspiration in searching for a bijection solving  Problem~\ref{problem.BDS} is Eq.~\eqref{continued fraction}, which motivates the exploration of the underlying combinatorial structures. In fact,  
 Claesson and Mansour~\cite{CM02} derived the first equation in \eqref{continued fraction} 
from a more general formula due to Clarke et al.~\cite{CSZ97}, and 
 both identities follow from Flajolet's Motzkin-path interpretation of general Jacobi-type continued fractions \cite{Fl80}, combined with the bijections $\Phi_{\FV}$ of Fran\c con and Viennot \cite{FV79}  or  $\Phi_{\FZ}$ of Foata and Zeilberger \cite{FZ90}. 
These bijections map permutations onto \emph{Laguerre histories}. Along these lines, the first two authors~\cite[Corollary 4.5]{CF23} resolved the problem by constructing the mapping\footnote{To be precise, here $\Phi_{\FV}$ refers to a variant of the original Fran\c con-Viennot bijection, sending each permutation of length $n$ to a (restricted) Laguerre history of length $n$; see the beginning of subsection~\ref{subsec:proof_of_theorem_ref_thm_baxter} for a further discussion on this.}
\begin{align}\label{def:phi}
\phi := \Phi^{-1}_{\FV} \circ \xi \circ \Phi_{\FV},
\end{align}
where $\xi$ is an involution on (restricted)  Laguerre histories.
As $\phi$ is clearly an involution on permutations,  if 
$(\rasc, \ldes)\sigma= (\rdes, \ldes)\phi(\sigma)$  for $\sigma\in \Sym_n$,  substituting 
$\sigma\to \phi(\sigma)$ we have 
\begin{equation}
(\rasc,\rdes, \ldes)\sigma= (\rdes, \rasc, \ldes)\phi(\sigma).
\end{equation}
Hence, the open problem and conjecture in \cite{BDS24} were resolved by the first two authors in \cite{CF23} via (restricted)  Laguerre histories, although the construction of $\xi$ and the resulting proofs are somewhat involved.  
  
In this paper we shall define a direct involution $\widehat{\phi}$  on permutations to resolve Problem~\ref{problem.BDS} and  Conjecture~\eqref{sym1}.  
In \cite{CSZ97},  Clarke et al.  managed to characterize 
 the composition $\Phi^{-1}_{\FZ}\circ \Phi_{\FV}$  by a  bijection $\Phi_{\CSZ}$ on permutations.
For our purpose we need a
variant $\Phi_{\SZ}$  of $\Phi_{\CSZ}$  in \cite{SZ10}. 
If $\sigma=\sigma_1\sigma_2\ldots \sigma_n\in \Sym_n$, let
\begin{align}\label{def:hattheta}
\widehat{\theta}(\sigma) &= \widehat{\sigma}_n\widehat{\sigma}_{n-1}\ldots\widehat{\sigma}_1,
\end{align}
where the map $\widehat{\phantom{a}}: a\mapsto \widehat{a}$ is defined by
$$\widehat{a} =\begin{cases}
n-a & \text{if $1\le a<n$,}\\
a & \text{if $a=n$.}
\end{cases}$$

Since the map $\widehat{\theta}$ is an involution on $\Sym_n$, it follows that the mapping
\begin{align}\label{eq:hatphi}
\widehat{\phi} = \Phi^{-1}_{\SZ} \circ \widehat{\theta} \circ \Phi_{\SZ}
\end{align}
is itself an involution on $\Sym_n$.

\begin{theorem}\label{thm:hatphi}
The mapping $\widehat{\phi}$ defined in \eqref{eq:hatphi} is an involution on $\Sym_n$ such that
for all $\sigma\in\Sym_n$,
\begin{align}\label{eq:phi213}
(\ldes, \rdes) \sigma=(\ldes, \rasc)\widehat{\phi}(\sigma).
\end{align} 
\end{theorem}
The mapping $\widehat{\phi}$ provides an answer to Problem~\ref{problem.BDS}.
Since $\widehat\phi$ is an involution, it follows that  for all $\sigma\in\Sym_n$,
\begin{align}\label{eq:phisymm}
(\ldes, \rasc,\rdes)\sigma = (\ldes, \rdes,\rasc)\widehat{\phi}(\sigma).
\end{align}
Clearly Eq.~\eqref{eq:phisymm} implies  \eqref{sym1}.

Let $\Sym_n(\baxter)$ denote the set of permutations in $\Sym_n$ that avoid simultaneously the two vincular patterns $3\underline{14}2$ and $2\underline{41}3$. These restricted  permutations are known as the \emph{Baxter permutations} \cite[Chapter 2.2.4]{Kit11}. For example, 436975128 is a Baxter permutation but 42173856 is not. The sequence $\{\Bax_n\}_{n\ge 0}$ that enumerates Baxter permutations of length $n$ is now called the sequence of \emph{Baxter numbers} and registered in the OEIS~\cite[A001181]{OEIS}. The following remarkable explicit formula  was first given by Chung, Graham, Hoggatt, and Kleiman~\cite[Eq.~(1) on p.~383]{CGHK78}, with an elegant bijective argument provided by Viennot~\cite{Vi81}.
\begin{align}
\Bax_n = \sum_{k=0}^{n-1}\frac{\binom{n+1}{k}\binom{n+1}{k+1}\binom{n+1}{k+2}}{\binom{n+1}{1}\binom{n+1}{2}}.
\end{align}

The trivariate symmetry shown in \eqref{eq:phisymm} concerns three length 3 vincular patterns. Interestingly enough, when vincular patterns of length 4 were taken into account (not from the permutation statistics point of view, but from the pattern avoidance point of view), we find a quite surprising application of $\widehat{\phi}$ in Baxter permutations.

\begin{theorem}\label{thm:baxter}
The mapping $\widehat\phi$ is closed on the set of Baxter permutations.
Consequently, in view of \eqref{eq:phisymm}, the following identity holds for all $n\ge 1$:
\begin{align}\label{eq:baxter equidist}
\sum_{\pi \in \Sym_n(\baxter)} q^{\ldes(\pi)}r^{\rasc(\pi)}s^{\rdes(\pi)}=
\sum_{\pi \in \Sym_n(\baxter)} q^{\ldes(\pi)}r^{\rdes(\pi)}s^{\rasc(\pi)}.
\end{align}
\end{theorem}

The rest of the paper is organized as follows. The proof of Theorem~\ref{thm:hatphi} is given in Section~\ref{sec:pf of Thm1.3}, where the definition of the mapping $\Phi_{\SZ}$ will be recalled, together with some property of $\Phi_{\SZ}$ that will be useful; see Lemma~\ref{lem:SZ}. Next in Section~\ref{sec:baxter}, we first establish the equivalence between the two mappings $\widehat\phi$ and $\phi$, then we prove Theorem~\ref{thm:baxter}, utilizing a characterization of Baxter permutations in terms of (restricted) Laguerre histories due to Viennot~\cite{Vi81}; see Lemma~\ref{lem:Viennot}. We end the paper with some closing thoughts.

%
%
%
%

\section{Proof of Theorem~\ref{thm:hatphi}}\label{sec:pf of Thm1.3}
First we associate to each permutation $\sigma \in \Sym_n$
a pictorial representation (Figure~\ref{fig.pictorial})
by placing vertices $1,2,\ldots,n$ along a horizontal axis
and then drawing an arc from $i$ to $\sigma_i$
above (resp.\ below) the horizontal axis
in case $\sigma_i > i$ (resp.\ $\sigma_i < i$);
if $\sigma_i = i$ we do not draw any arc.
\begin{figure}[t]
\centering
\vspace*{4cm}
\begin{picture}(60,20)(120, -65)
\setlength{\unitlength}{2mm}
\linethickness{.5mm}
\put(-2,0){\line(1,0){54}}
\put(0,0){\circle*{1,3}}\put(0,0){\makebox(0,-6)[c]{\small 1}}
\put(5,0){\circle*{1,3}}\put(5,0){\makebox(0,-6)[c]{\small 2}}
\put(10,0){\circle*{1,3}}\put(10,0){\makebox(0,-6)[c]{\small 3}}
\put(15,0){\circle*{1,3}}\put(15,0){\makebox(0,-6)[c]{\small 4}}
\put(20,0){\circle*{1,3}}\put(20,0){\makebox(0,-6)[c]{\small 5}}
\put(25,0){\circle*{1,3}}\put(25,0){\makebox(0,-6)[c]{\small 6}}
\put(30,0){\circle*{1,3}}\put(30,0){\makebox(0,-6)[c]{\small 7}}
\put(35,0){ \circle*{1,3}}\put(36,0){\makebox(0,-6)[c]{\small 8}}
\put(40,0){\circle*{1,3}}\put(40,0){\makebox(0,-6)[c]{\small 9}}
\put(45,0){\circle*{1,3}}\put(45,0){\makebox(0,-6)[c]{\small 10}}
\put(50,0){\circle*{1,3}}\put(50,0){\makebox(0,-6)[c]{\small 11}}
\green{\qbezier(0,0)(20,14)(40,0)
\qbezier(40,0)(42.5,6)(45,0)}
\red{\qbezier(4,0)(6.5,5)(9,0)
\qbezier(9,0)(18,10)(29,0)}
\blue{\qbezier(18,0)(20.5,5)(23.5,0)
\qbezier(23.2,0)(36,12)(48.5,0)
\qbezier(18,0)(34,-12)(48.5,0)}
\red{\qbezier(2.5,0)(17,-14)(27.5,0)}
\green{\qbezier(-3,0)(22,-20)(42,0)}
\put(-5,-13){$\sigma = 9~3~7~4~6~11~2~8~10~1~5 = (1,9,10)\,(2,3,7)\,(4)\,(5,6,11)\,(8)$}
\end{picture}
\caption{
   The pictorial representation of a permutation $\sigma$ in $\Sym_{11}$ \label{fig.pictorial}
}
\end{figure}
Each vertex thus has either
out-degree = in-degree = 1 (if it is not a fixed point) or
out-degree = in-degree = 0 (if it is a fixed point). Of course, the arrows on the arcs are redundant,
because the arrow on an arc above (resp.\ below) the axis
always points to the right (resp.\ left).

Next we introduce various kinds of permutation statistics, both linear and cyclic ones. Throughout the following definitions we fix a given permutation $\sigma\in\Sym_n$. First recall two cyclic statistics, namely, the crossing numbers and the nesting numbers. In parallel with linear statistics, we introduce them as the sums of the following coordinate statistics. For $1\le i\le n$, let
\begin{align*}
\cros(i,\sigma) &:= |\{j\in [n]:j<i\le\sigma_j<\sigma_i\text{ or }j>i>\sigma_j>\sigma_i\}|, \\
\nest(i,\sigma) &:= |\{j\in [n]:j<i\le\sigma_i<\sigma_j\text{ or }j>i>\sigma_i>\sigma_j\}|, \\
\cros(\sigma) &:= \sum_{i=1}^n\cros(i,\sigma),\quad
\nest(\sigma) := \sum_{i=1}^n\nest(i,\sigma).
\end{align*}
For example, if
   $\sigma = 9\,3\,7\,4\,6\,11\,2\,8\,10\,1\,5
           \in \Sym_{11}$, then $\sigma= (1,9,10)\,(2,3,7)\,(4)\,(5,6,11)\,(8)$ and a pictorial representation of $\sigma$  is given in Figure~\ref{fig.pictorial} with
           $\cros(\sigma)=7$ and $\nest(\sigma)=10$.

We further distinguish three kinds of crossings, as well as three kinds of nestings. Namely the {\it ending-crossing} number, the {\it upper-crossing} number, the {\it lower-crossing} number, the {\it ending-nesting} number, the {\it upper-nesting} number, and the {\it lower-nesting} number of $\sigma$ are denoted and defined respectively as follows.
\begin{align}
\label{def:ecr}
\ecr(\sigma) &:= |\{(i,j)\in[n]\times[n]:i<j\le\sigma_i<\sigma_j=n\}|,\\
\ucr(\sigma) &:= |\{(i,j)\in[n]\times[n]:i<j\le\sigma_i<\sigma_j<n\}|,\\
\lcr(\sigma) &:= |\{(i,j)\in[n]\times[n]:i>j>\sigma_i>\sigma_j\}|,\\
\ene(\sigma) &:= |\{(i,j)\in[n]\times[n]:i<j\le\sigma_j<\sigma_i=n\}|,\\
\une(\sigma) &:= |\{(i,j)\in[n]\times[n]:i<j\le\sigma_j<\sigma_i<n\}|,\\
\lne(\sigma) &:= |\{(i,j)\in[n]\times[n]:i>j>\sigma_j>\sigma_i\}|.
\end{align}
Clearly we have
\begin{align}\label{eq:crosdecomp}
\cros(\sigma)&=\ecr(\sigma)+\ucr(\sigma)+\lcr(\sigma),\\
\nest(\sigma)&=\ene(\sigma)+\une(\sigma)+\lne(\sigma).
\label{eq:nestdecomp}
\end{align}
We also need the following variant of $\ene$:
\begin{align*}
\wene(\sigma) &:= |\{(i,j)\in[n]\times[n]:\sigma_j<j<i\le\sigma_i=n\}|,
\end{align*}
which leads to a variant of the nesting numbers:
\begin{align}\label{eq:wnedecomp}
\wne(\sigma) &:= \wene(\sigma)+\une(\sigma)+\lne(\sigma).
\end{align}
The following set-valued versions of $\ene(\sigma)$ and $\wene(\sigma)$ are also required.
\begin{align*}
\Ene(\sigma) &:=\{j\in[n]:i<j\le\sigma_j<\sigma_i=n\},\\
\Wene(\sigma) &:=\{j\in[n]:\sigma_j<j<i\le\sigma_i=n\}.
\end{align*}

For linear statistics, we need the following pair of set-valued statistics, which are related to the sets of ascent bottoms and descent tops, respectively, and were introduced in \cite{CF23}:
\begin{align*}
\Aba(\sigma) &=\{\sigma_i : \sigma_n<\sigma_i<\sigma_{i+1}\}, \\[3pt]
\Dtb(\sigma) &=\{\sigma_i : \sigma_{i+1}<\sigma_i <\sigma_n\}.
\end{align*}

The next relation is easy to verify by the definitions of the statistics and was noted in the proof of Prop.~4.2 in \cite{CF23}. For any $\pi\in\Sym_n$,
\begin{align}
\label{eq:diff-213-231}
\rasc(\pi)=\rdes(\pi)-|\Aba(\pi)|+|\Dtb(\pi)|.
\end{align}

To make this paper self-contained, we include here a description of the mapping $\Phi_{\SZ}$ in \cite{SZ10, SZ12} as in the following. The reader may go through this process with the permutation provided in Example \ref{ex2} below.

Given a permutation $\sigma=\sigma_1\cdots \sigma_n$, we proceed as follows:
    \begin{itemize}
    \item Determine the sets of descent tops $F$  and descent bottoms $F'$, and their complements $G$ and $G'$, respectively;
    \item let $f$ and  $g$ be  the  increasing permutations of $F$ and $G$, respectively;
    \item construct the biword $\binom{f}{f'}$: for each $j$ in the first row $f$ starting from the \textbf{smallest} (leftmost), the entry in $f'$ that below $j$ is the $(\ldes(j,\sigma)+1)$-th largest entry of $F'$ that is smaller than $j$ and not yet chosen;
    \item construct the biword $\binom{g}{g'}$: for each $j$ in the first row  $g$ starting from the \textbf{largest} (rightmost), the entry  below $j$ in $g'$ is the $(\ldes(j,\sigma)+1)$-th smallest entry of $G'$ that is not smaller than $j$ and not yet chosen;
    \item rearrange the columns so that the top row is in increasing order,
    we obtain the permutation $\tau=\Phi_{\SZ}(\sigma)$ as the bottom row of the rearranged biword.
\end{itemize}

\begin{example}\label{ex2}
For  $\sigma=4~7~1~8~6~3~2~5$ with
$$
\begin{tabular}{c|c|c|c|c|c|c|c|c}
$\sigma=$&4&7&1&8&6&3&2&5\\
\hline
$\ldes(i, \sigma)$&0&0&0&0&1&1&1&2\\
\hline
$\rdes(i,\sigma)$&2&1&0&0&0&0&0&0\\
\end{tabular}
\quad
\xrightarrow{\Phi_{\SZ}} \quad
\begin{tabular}{c|c|c|c|c|c|c|c|c}
$\tau=$&5&7&1&4&8&2&6&3\\
\hline
$\cros(i,\tau)$&2&0&0&0&0&1&1&1\\
\hline
$\nest(i,\tau)$&0&1&0&2&0&0&0&0\\
\end{tabular}.
$$
We have
$$
	\binom{f}{f'}
	=\left(\begin{array}{cccc}
    3_1 & 6_1 & 7 & 8 \\
    1 & 2 & 6 & 3
    \end{array}
    \right),
	\quad
	\binom{g}{g'}
	= \left(\begin{array}{cccc}
    1 & 2_1 & 4 & 5_2 \\
    5 & 7 & 4 & 8
    \end{array}
	\right).
	$$
	Hence
	$$
	\tau
	= \left(\begin{array}{cc} f & g \\ f' & g'\end{array} \right)
	\xrightarrow{\text{rearrange}} \left(\begin{array}{cccccccc}
    1 & 2 & 3 & 4 & 5 & 6 & 7 & 8 \\
    5 & 7 & 1 & 4 & 8 & 2 & 6 & 3
    \end{array}
	\right).
	$$
\end{example}

The mapping $\Phi_{\SZ}$ transforms the concerned permutation statistics as follows.
\begin{lemma}\label{lem:SZ}
For $n\ge 1$, $\Phi_{\SZ}$ is a bijection on $\Sym_n$ such that for all $\pi\in\Sym_n$ and its image $\sigma:=\Phi_{\SZ}(\pi)$, we have for $1\le i\le n-1$,
\begin{align}
\label{eq:ndes-wex}
\pi_i<\pi_{i+1} &\Longleftrightarrow \sigma_{\pi_i}\ge \pi_i\neq\pi_n,
\end{align}
and
\begin{align}
\label{eq:last entry}
\sigma_{\pi_n} &=n,\\
\label{eq:312-cros}
(\rdes,\ldes) \pi &= (\nest,\cros) \sigma,\\
(\Aba,\Dtb) \pi &= (\Ene,\Wene) \sigma,\label{eq:aba-ene}\\
\rasc(\pi) &= \wne(\sigma).
\label{eq:213-wne}
\end{align}
\end{lemma}
\begin{proof}
Note that \eqref{eq:ndes-wex} and \eqref{eq:312-cros} are contained in \cite[Eq.~(31)]{SZ10} and \cite[Eq.~(30)]{SZ10}, respectively. In fact, the proof given in \cite{SZ10} reveals the following stronger results that imply \eqref{eq:312-cros} upon summation:
$$\rdes(i,\pi)=\nest(i,\sigma),\quad \ldes(i,\pi)=\cros(i,\sigma),$$
for all $1\le i\le n$. Next, \eqref{eq:last entry} follows from the fact that $\Phi_{\CSZ}$ preserves the last entry of a permutation. Consequently, it suffices to prove \eqref{eq:aba-ene} and \eqref{eq:213-wne}.

Take any $\pi_i\in\Aba(\pi)$, then by \eqref{eq:ndes-wex} and \eqref{eq:last entry} we have $$\pi_n<\pi_i\le\sigma_{\pi_i}<\sigma_{\pi_n}=n,$$
thus $\pi_i\in\Ene(\sigma)$. The same argurment works for the opposite direction so we have $\Aba(\pi)=\Ene(\sigma)$. Similarly, for any $\pi_i\in\Dtb(\pi)$, applying \eqref{eq:ndes-wex} and \eqref{eq:last entry} we see that
$$n=\sigma_{\pi_n}\ge \pi_n>\pi_i>\sigma_{\pi_i},$$
i.e., $\pi_i\in\Wene(\sigma)$. Vice versa, $\pi_i\in\Wene(\sigma)$ implies that $\pi_i\in\Dtb(\pi)$. This completes the proof of \eqref{eq:aba-ene}.

Next, we compute the following difference using \eqref{eq:diff-213-231}, \eqref{eq:aba-ene}, \eqref{eq:nestdecomp}, and \eqref{eq:wnedecomp},
$$\rdes(\pi)-\rasc(\pi)=|\Aba(\pi)|-|\Dtb(\pi)|=\ene(\sigma)-\wene(\sigma)=\nest(\sigma)-\wne(\sigma).$$
This is equivalent to \eqref{eq:213-wne} in view of \eqref{eq:312-cros}.
\end{proof}

The following theorem contains the main
 properties of the mapping $\widehat{\theta}$.
\begin{theorem}\label{thm:sextuple}
The mapping $\hat{\theta}$ is a bijection on $\Sym_n$ such that for all permutation $\sigma\in\Sym_n$ we have
\begin{align}\label{eq:sextuple}
(\ecr,\ucr,\lcr,\ene,\une,\lne)\: \sigma &= (\ecr,\lcr,\ucr,\wene,\lne,\une)\: \widehat{\theta}(\sigma).
\end{align}
In view of the three decompositions \eqref{eq:crosdecomp}--\eqref{eq:wnedecomp} we have
\begin{align}\label{eq:crosequi}
(\cros, \nest)\sigma=(\cros, \wne)\hat{\theta}(\sigma).
\end{align}
\end{theorem}
The proof of Theorem~\ref{thm:sextuple} is a bit involved, so we decide to first show the proof of our main theorem, utilizing all the results we have collected so far.

\begin{proof}[Proof of Theorem~\ref{thm:hatphi}]
Note that \eqref{eq:hatphi} is equivalent to $\Phi_{\SZ}\circ\widehat{\phi}=\hat{\theta}\circ\Phi_{\SZ}$.
Applying this to a given permutation $\pi$ and writing $\sigma:=\Phi_{\SZ}(\pi)$, we obtain
\begin{align}\label{eq:long}
\Phi_{\SZ}(\widehat{\phi}(\pi)) &= \widehat{\theta}(\sigma).
\end{align}
Furthermore, we deduce from \eqref{eq:312-cros}, \eqref{eq:213-wne}, and \eqref{eq:long} that
\begin{align*}
\ldes(\widehat{\phi}(\pi)) &=\cros(\Phi_{\SZ}(\widehat{\phi}(\pi)))=\cros(\widehat{\theta}(\sigma)),\\
\rasc(\widehat{\phi}(\pi)) &=\wne(\Phi_{\SZ}(\widehat{\phi}(\pi)))=\wne(\widehat{\theta}(\sigma)).
\end{align*}
Combining the above two relations with
\eqref{eq:crosequi} and \eqref{eq:312-cros},  we have
\begin{align*}
(\ldes, \rasc)\widehat{\phi}(\pi)
=(\cros, \widetilde{\nest})\widehat{\theta}(\sigma)
=(\cros, \nest)\sigma
=(\ldes, \rdes)\pi.
\end{align*}
This is \eqref{eq:phi213} and it completes the proof of Theorem~\ref{thm:hatphi}.
\end{proof}

\begin{proof}[Proof of Theorem~\ref{thm:sextuple}]
We first show that $\hat{\theta}$ swaps the pair $(\ucr,\lcr)$. Suppose $(i,j)\in[n]\times[n]$ is an upper-crossing pair in $\sigma$, i.e., $i<j\le\sigma_i<\sigma_j<n$. Then
$$n+1-i>n+1-j>\hat{\sigma}_i>\hat{\sigma}_j.$$
Note that (see \eqref{def:hattheta}) $\hat{\sigma}_i$ (resp.~$\hat{\sigma}_j$) is the ($n+1-i$)-th (resp.~($n+1-j$)-th) letter in $\hat{\theta}(\sigma)$, which means that $(n+1-j,n+1-i)$ forms a lower-crossing pair in $\hat{\theta}(\sigma)$. Obviously this process is reversible, so the number of upper-crossings in $\sigma$ equals the number of lower-crossings in $\hat{\theta}(\sigma)$. In the same vein, we can show that the number of lower-crossings in $\sigma$ equals the number of upper-crossings in $\hat{\theta}(\sigma)$. It follows that
$$(\ucr,\lcr)\: \sigma = (\lcr,\ucr)\: \hat{\theta}(\sigma).$$

It remains to show $\ecr(\sigma)=\ecr(\hat{\theta}(\sigma))$. If $\sigma_n=n$, then $\hat{\theta}(\sigma)$ begins with $n$, so we have $\ecr(\sigma)=0=\ecr(\hat{\theta}(\sigma))$. Otherwise we can assume $\sigma_j=n$ with $j<n$, and we decompose the interval $[j,n]:=\{j,j+1,\ldots,n\}$ in two ways involving four subsets. Firstly, viewing elements in $[j,n]$ as images under $\sigma$, we get the split\footnote{We use $A\uplus B$ to denote the disjoint union of two sets $A$ and $B$.}:
\begin{align}\label{eq:split-1}
[j,n] &= E_1(\sigma)\uplus E_3(\sigma),
\end{align}
where
\begin{align*}
E_1(\sigma) &:= \{\sigma_i:i<j\le\sigma_i\},\\
E_3(\sigma) &:= \{\sigma_i:j\le\min(i,\sigma_i)\}.
\end{align*}
Secondly, viewing elements in $[j,n]$ as preimages under $\sigma$, we get another split:
\begin{align}\label{eq:split-2}
[j,n] &= E_2(\sigma)\uplus E_4(\sigma),
\end{align}
where
\begin{align*}
E_2(\sigma) &:= \{i:\sigma_i< j\le i\},\\
E_4(\sigma) &:= \{i:j\le\min(i,\sigma_i)\}.
\end{align*}
Now each pair $(i,\sigma_i)$ satisfying $j\le\min(i,\sigma_i)$ corresponds to a unique element in $E_3(\sigma)$ and $E_4(\sigma)$, respectively, hence we have $|E_3(\sigma)|=|E_4(\sigma)|$. Taking cardinalities in both \eqref{eq:split-1} and \eqref{eq:split-2}, we derive that
\begin{align}
|E_1(\sigma)| &=|E_2(\sigma)|.\label{eq:E1=E2}
\end{align}
We further claim that
\begin{align}
\ecr(\sigma) &= |E_1(\sigma)|,\label{eq:E1}\\
\ecr(\hat{\theta}(\sigma)) &= |E_2(\sigma)|.\label{eq:E2}
\end{align}

Combining \eqref{eq:E1=E2}, \eqref{eq:E1}, and \eqref{eq:E2}, we conclude that $\ecr(\sigma)=\ecr(\hat{\theta}(\sigma))$, as desired. Now \eqref{eq:E1} is clear from the definition of ending-crossing number (see \eqref{def:ecr}), since $\sigma_i\in E_1(\sigma)$ if and only if $(i,j)$ forms an ending-crossing pair of $\sigma$. For \eqref{eq:E2}, take any ending-crossing of $\hat{\theta}(\sigma)$, say $(n+1-k,n+1-j)$ satisfying
$$n+1-k<n+1-j\le \hat{\sigma}_k<\hat{\sigma}_{n+1-j}=n,$$
we have
$$k>j>n-\hat{\sigma}_k=\sigma_k,$$
which implies that $k\in E_2(\sigma)$. Conversely, any element $k\in E_2$ corresponds to $n+1-k$, which forms the ending-crossing pair $(n+1-k,n+1-j)$ in $\hat{\theta}(\sigma)$ using similar argument. This one-to-one correspondence proves \eqref{eq:E2}, and thus completes the proof of the first three identities contained in \eqref{eq:sextuple} that involve crossing numbers.

Now for the nesting numbers, a similar argument gives us that $(\une,\lne)\sigma=(\lne,\une)\hat{\theta}(\sigma)$. For the final missing piece in \eqref{eq:sextuple}, i.e., $\ene(\sigma)=\wene(\hat{\theta}(\sigma))$, we can prove the following stronger result:
$$j\in \Ene(\sigma)\Longleftrightarrow n+1-j\in \Wene(\hat{\theta}(\sigma)).$$
Indeed, if $j\in\Ene(\sigma)$, we have $i<j\le\sigma_j<\sigma_i=n$, which is equivalent to
$$n=\hat{\sigma}_i\ge n+1-i>n+1-j>\hat{\sigma}_j.$$
I.e., $n+1-j\in\Wene(\hat{\theta}(\sigma))$. Conversely, $n+1-i\in\Wene(\hat{\theta}(\sigma))$ implies $j\in\Ene(\sigma)$ as well.
\end{proof}

\section{Further application of \texorpdfstring{$\widehat\phi$}{widehat-phi}}\label{sec:baxter}

\subsection{Link to the past} 
A \emph{$2$-Motzkin path} of length $n$ is a lattice path in the first quadrant starting from $(0,0)$, ending at $(n,0)$, with four kinds of steps: up steps ($\rU$) going from $(a,b)$ to $(a+1,b+1)$; down steps ($\rD$) going from $(a,b)$ to $(a+1,b-1)$; and level stpes going from $(a,b)$ to $(a+1,b)$ with two possible colors red ($\rLr$) and blue ($\rLb$). We often express $2$-Motzkin paths as words consisted of letters from $\{\rU,\rD,\rLr,\rLb\}$. For convenience, we will also label a step as $\ULr$ (resp.~$\DLb$), if it is either an up (resp.~a down) step or a level step in red (resp.~blue). 

Next we recall the notion of \emph{shifted and restricted Laguerre history}, which is a variant of restricted Laguerre history used by the first two authors in \cite[Definition 2.1]{CF23}.

\begin{Def}\label{def:LH}
A (shifted and restricted) Laguerre history of length $n$ is a triple $W=(w,h,c)$, such that
\begin{enumerate}
  \item $w=w_1\cdots w_n$ is a $2$-Motzkin path of length $n$.
  \item $h=h_1\cdots h_n$ with $h_i:=\#\{j : j<i, w_j=\rU\}-\#\{j : j<i, w_j=\rD\}$.
  \item $c=c_1\cdots c_n$ is a sequence of integers with $h_i\ge c_i\ge \begin{cases}
  0, & \text{if } w_i=\ULr,\\
  1, & \text{if } w_i=\DLb.
  \end{cases}$
\end{enumerate}
For each $i=1,2,\ldots,n$, we say the $i$-th step of $W$ is of type $w_i$, height $h_i$, and weight $c_i$. We refer to $W$ as an sr-Laguerre history and denote by $\LH_n$ the set of sr-Laguerre histories of length $n$.
\end{Def}

\begin{Def}[{\cite[Definition 3.1]{CF23}}]
Given an sr-Laguerre history $W=(w,h,c)\in\LH_n$, we define its {\it critical step} to be the last step of $W$ that is weighted by $0$. We use
\begin{align}
\cs(W)& :=\max\{i\in[n] : c_i=0\}
\end{align}
to denote the index of the critical step of $W$.
\end{Def}

\begin{example}
Displayed below is an sr-Laguerre history of length $n=10$, where the weight sequence $c=c_1\cdots c_{10}$ is marked over each step and the critical step is marked by a red cross.

\begin{figure}[!ht]
  \begin{center}
    \begin{tikzpicture}[scale=0.8]
      \foreach\i\j in {0/0, 1/1, 2/2, 3/2, 4/2, 5/1, 6/1, 7/0, 8/0, 9/1, 10/0}
      {\filldraw[black] (\i,\j) circle[radius=2.5pt];}
      \draw[line width=1.7pt] (0,0) -- (1,1) -- (2,2);
      \draw[line width=1.7pt] (4,2) -- (5,1);
      \draw[line width=1.7pt] (6,1) -- (7,0);
      \draw[line width=1.7pt] (8,0) -- (9,1) -- (10,0);
      \draw[line width=1.7pt,draw=cyan] (2,2) -- (3,2);
      \draw[line width=1.7pt,draw=cyan] (5,1) -- (6,1);
      \draw[line width=1.7pt,draw=red] (3,2) -- (4,2);
      \draw[line width=1.7pt,draw=red] (7,0) -- (8,0);
      \node[above left] at (.6,.4) {$0$};
      \node[above left] at (1.6,1.4) {$1$};
      \node[above] at (2.5,2) {$2$};
      \node[above] at (3.5,2) {$1$};
      \node[above] at (5.5,1) {$1$};
      \node[above] at (7.5,0) {$0$};
      \node[above left] at (8.6,.4) {$0$};
      \node[above right] at (4.4,1.4) {$2$};
      \node[above right] at (6.4,.4) {$1$};
      \node[above right] at (9.4,.4) {$1$};
      \draw[color=red] plot[only marks,mark=x,mark size=4pt] (8.7,.5);
    \end{tikzpicture}
    \caption{An sr-Laguerre history $W\in\LH_{10}$}
    \label{fig:history}
  \end{center}
\end{figure}
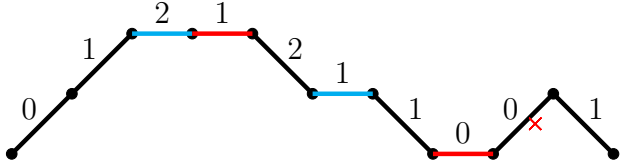

\end{example}

To make the paper self-contained, we also recall here the definition of the involution $\xi:\LH_n\to \LH_n$ given in \cite[Theorem 3.2]{CF23}, as well as a table that lists out all the cases in the construction of $\xi$. This Table~\ref{xsi-cases} will be quite useful when we show a key property of $\xi$ in Lemma~\ref{lem:xi closed}.

\begin{Def}\label{def:xsi}
Let $\xi$ be the unique involution $\xi:\LH_n\rightarrow \LH_n$ such that, if $W=(w,h,c)$ with $\cs(W)=m$ and $V:=\xi(W)=(v,g,b)$, then
\begin{enumerate}
  \item $\cs(V)=n+1-m$.
  \item For any $n+1-m\neq j\in [n]$, $v_j=\ULr$ if and only if $w_{n+1-j}=\DLb$.
  \item For any $j\in[n]$, $g_j=
  \begin{cases}
  h_{n+1-j}+1, & \text{if $j>n+1-m$ and $v_j=\DLb$,} \\
  h_{n+1-j}-1, & \text{if $j<n+1-m$ and $v_j=\ULr$,} \\
  h_{n+1-j}, & \text{otherwise}.
  \end{cases}$
  \item For any $j\in[n]$, $b_j=g_j-h_{n+1-j}+c_{n+1-j}$.
\end{enumerate}
\end{Def}

\renewcommand{\arraystretch}{1.4}
\begin{table}
{\small
\begin{tabular}{cccccccc}
\toprule
Case no. & $j$ & $v_j$ & $v_{j+1}$ & $w_{n+1-j}$ & $w_{n-j}$ & $g_j$ & $g_{j+1}$\\
\midrule
1 & $n+1-m$ & $\ULr$ & $\ULr$ & $\ULr$ & $\DLb$ & $h_{n+1-j}$ & $h_{n-j}$ \\

2 & $n+1-m$ & $\ULr$ & $\DLb$ & $\ULr$ & $\ULr$ & $h_{n+1-j}$ & $h_{n-j}+1$ \\

3 & $n-m$ & $\ULr$ & $\ULr$ & $\DLb$ & $\ULr$ & $h_{n+1-j}-1$ & $h_{n-j}$ \\

4 & $n-m$ & $\DLb$ & $\ULr$ & $\ULr$ & $\ULr$ & $h_{n+1-j}$ & $h_{n-j}$ \\

5 & $<n-m$ & $\ULr$ & $\ULr$ & $\DLb$ & $\DLb$ & $h_{n+1-j}-1$ & $h_{n-j}-1$ \\

6 & $<n-m$ & $\ULr$ & $\DLb$ & $\DLb$ & $\ULr$ & $h_{n+1-j}-1$ & $h_{n-j}$ \\

7 & $<n-m$ & $\DLb$ & $\ULr$ & $\ULr$ & $\DLb$ & $h_{n+1-j}$ & $h_{n-j}-1$ \\

8 & $<n-m$ & $\DLb$ & $\DLb$ & $\ULr$ & $\ULr$ & $h_{n+1-j}$ & $h_{n-j}$ \\

9 & $>n+1-m$ & $\ULr$ & $\ULr$ & $\DLb$ & $\DLb$ & $h_{n+1-j}$ & $h_{n-j}$ \\

10 & $>n+1-m$ & $\ULr$ & $\DLb$ & $\DLb$ & $\ULr$ & $h_{n+1-j}$ & $h_{n-j}+1$ \\

11 & $>n+1-m$ & $\DLb$ & $\ULr$ & $\ULr$ & $\DLb$ & $h_{n+1-j}+1$ & $h_{n-j}$ \\

12 & $>n+1-m$ & $\DLb$ & $\DLb$ & $\ULr$ & $\ULr$ & $h_{n+1-j}+1$ & $h_{n-j}+1$ \\

13 & $n \: (m=1)$ & $\rLr$ & n/a & $\ULr$ & n/a & 0 & 0 \\

14 & $n \: (m>1)$ & $\rD$ & n/a & $\ULr$ & n/a & 1 & 0  \\
\bottomrule\\
\end{tabular}
}
\caption{All cases in the construction of $V=(v,g,b)=\xi(W)$, where $\cs(W)=m$.}\label{xsi-cases}
\end{table}

To end this subsection, we show that actually $\widehat\phi$ is an alternative way of defining $\phi$. For clearer notation, we use $\iota:\pi\mapsto\pi^{-1}$ to denote the map of taking group-theoretical inverse, defined on $\Sym_n$. First off, it is known that (see for example \cite{CSZ97}, the paragraph above Corollary 4.13 in \cite{CF23}, and note that $\Phi_{\CSZ}=\iota\circ\Phi_{\SZ}$ \cite[p.~1696]{SZ10})
\begin{align}\label{eq:FVFZ}
\Phi_{\FV}=\Phi_{\FZ}\circ\iota\circ\Phi_{\SZ}.
\end{align}

Moreover, it was shown in \cite[Eq.~(4.34)]{CF23} that
\begin{align}\label{eq:xitheta}
\Phi_{\FZ}^{-1}\circ\xi\circ\Phi_{\FZ} = \theta,
\end{align}
where $\theta:\Sym_n\to \Sym_n$ is an involution defined as follows: if $\pi=\pi_1\pi_2\ldots\pi_n\in\Sym_n$, then
$$\theta(\pi):=\pi^c_{n-1}\pi^c_{n-2}\ldots\pi^c_1\pi^c_n,$$
where $\pi^c_i=n+1-\pi_i$ for $i\in[n]$.

\begin{proposition} Let $\widehat{\theta}$ be defined by \eqref{def:hattheta}, 
then
\begin{align}\label{hatphi=phi}
\phi=\Phi_{\SZ}^{-1}\circ\widehat{\theta}\circ\Phi_{\SZ}=\widehat{\phi}.
\end{align}
\end{proposition}
\begin{proof}   We first verify  
\begin{align}\label{eq:hattheta}
\widehat{\theta}=\iota\circ\theta\circ\iota.
\end{align}
The proof is essentially an entry-by-entry verification. If a permutation $\pi$ sends $a$ to $b$, we record this as $(a,b)_{\pi}$. Now given any permutation $\sigma\in\Sym_n$, assume that $\sigma_j=n$. One sees that
$$(j,n)_{\sigma}\to (n,j)_{\iota(\sigma)}\to (n,n+1-j)_{\theta(\iota(\sigma))}\to (n+1-j,n)_{\iota\circ\theta\circ\iota(\sigma)},$$
which agrees with \eqref{def:hattheta}. While for any $i\neq j$, $1\le i\le n$, we have  $\sigma_i\neq n$ and
$$(i,\sigma_i)_{\sigma}\to (\sigma_i,i)_{\iota(\sigma)}\to (n-\sigma_i,n+1-i)_{\theta(\iota(\sigma))}\to (n+1-i,n-\sigma_i)_{\iota\circ\theta\circ\iota(\sigma)},$$
which agrees with \eqref{def:hattheta} as well. 
Plugging \eqref{eq:FVFZ}, \eqref{eq:xitheta}, and \eqref{eq:hattheta} into \eqref{def:phi}, we obtain \eqref{hatphi=phi}.
\end{proof}

\subsection{Proof of Theorem~\ref{thm:baxter}}\label{subsec:proof_of_theorem_ref_thm_baxter}
Thanks to \eqref{hatphi=phi}, proving Theorem~\ref{thm:baxter} is now equivalent to showing that the involution $\phi=\Phi_{\FV}^{-1}\circ\xi\circ\Phi_{\FV}$ is closed over the set of Baxter permutations. Here $\Phi_{\FV}$ refers to the version of Fran\c con-Viennot bijection used in \cite[p.~12--13]{CF23}. We recall here the definition of its inverse $\Phi_{\FV}^{-1}$ by running through an example. 

Given an sr-Laguerre history $W=(w,h,c)\in\LH_n$, we explain how to recover its preimage $\sigma\in\Sym_n$. For each $i=1,2,\ldots,n$, we let $\sigma^{(i)}$ be the subword obtained from $\sigma$ by deleting all letters in $\sigma$ that are greater than $i$, and we call it the $i$-th \emph{section} of $\sigma$. Let $\sigma^{(0)}$ be the empty permutation. We scan the $2$-Motzkin path $w$ from left to right. For $i=1,2,\ldots,n$, when the $i$-th step in $w$ is passed, we insert the letter $i$, possibly with open slots ``\varspace[.5em]'' attached to it, into the $c_i$-th open slot in $\sigma^{(i-1)}$ to get a new permutation $\sigma^{(i)}$. Here the slots are always labeled from right to left, starting with the label $0$. The four types of $w_i$ dictate the insertion types as follows.
\begin{align*}
  w_i=
    \begin{cases}
      \rU, & \text{ insert \varspace~$i$~\varspace, }\\
      \rLr, & \text{ insert $i$~\varspace, }\\
      \rLb, & \text{ insert \varspace~$i$, }\\
      \rD, & \text{ insert $i$. }\\
    \end{cases}
\end{align*}
Because of the shifting ($1\le c_i\le h_i$ rather than $0\le c_i\le h_i-1$ for each step $w_i=\DLb$), we see that the final permutation $\sigma^{(n)}$ must end with an open slot and we take it to be the preimage $\sigma:=\sigma^{(n)}$. 

For example, the preimage $\sigma$ of the history $W=(w,h,c)\in\LH_{10}$ shown in Figure~\ref{fig:history} is constructed step-by-step as follows. 
\begin{align*}
&w_1=\rU,~c_1=0, & \Rightarrow \qquad & \sigma^{(1)}= \text{\varspace~$1$~\varspace}; \\
&w_2=\rU,~c_2=1, & \Rightarrow \qquad & \sigma^{(2)}= \text{\varspace~$2$~\varspace~$1$~\varspace}; \\
&w_3=\rLb,~c_3=2, & \Rightarrow \qquad & \sigma^{(3)}= \text{\varspace~$3~2$~\varspace~$1$~\varspace}; \\
&w_4=\rLr,~c_4=1, & \Rightarrow \qquad & \sigma^{(4)}= \text{\varspace~$3~2~4$~\varspace~$1$~\varspace}; \\
&w_5=\rD,~c_5=2, & \Rightarrow \qquad & \sigma^{(5)}= \text{$5~3~2~4$~\varspace~$1$~\varspace}; \\
&w_6=\rLb,~c_6=1, & \Rightarrow \qquad & \sigma^{(6)}= \text{$5~3~2~4$~\varspace~$6~1$~\varspace}; \\
&w_7=\rD,~c_7=1, & \Rightarrow \qquad & \sigma^{(7)}= \text{$5~3~2~4~7~6~1$~\varspace}; \\
&w_8=\rLr,~c_8=0, & \Rightarrow \qquad & \sigma^{(8)}= \text{$5~3~2~4~7~6~1~8$~\varspace}; \\
&w_9=\rU,~c_9=0, & \Rightarrow \qquad & \sigma^{(9)}= \text{$5~3~2~4~7~6~1~8$~\varspace~$9$~\varspace}; \\
&w_{10}=\rD,~c_{10}=1, & \Rightarrow \qquad & \sigma:=\sigma^{(10)}= \text{$5~3~2~4~7~6~1~8~10~9$~\varspace}.
\end{align*}
We note in passing that this preimage $\sigma = 5~3~2~4~7~6~1~8~10~9$ avoids both patterns $3\underline{14}2$ and $2\underline{41}3$. Viennot \cite{Vi81} (see also \cite[Chapter 4b]{ABjC}) characterized the Baxter permutations by the weight sequences $c$ of their associated Laguerre histories. We adapt his characterization and recast it below in terms of sr-Laguerre histories rather than (unrestricted) Laguerre histories.

\begin{Def}\label{def:prudent}
An sr-Laguerre history $W=(w,h,c)\in\LH_n$ is called \emph{prudent}, if for each pair of consecutive steps $(w_i,w_{i+1})$, $1\le i<n$, the difference between their weights $c_{i+1}-c_{i}$ takes only two values according to the type of $w_i$. Namely, $W$ satisfies that:
\begin{enumerate}
  \item if $w_i=\ULr$, then $c_{i+1}-c_i$ equals $0$ or $1$;
  \item if $w_i=\DLb$, then $c_{i+1}-c_i$ equals $0$ or $-1$.
\end{enumerate}
The set of all prudent sr-Laguerre histories of length $n$ is denoted as $\LH^*_n$.
\end{Def}

\begin{lemma}[\cite{Vi81}, {\cite[Chapter 4b, p.~111]{ABjC}}]\label{lem:Viennot}
Given a permutation $\sigma\in\Sym_n$ and its image $W:=\Phi_{\FV}(\sigma)\in\LH_n$, we have that
$$\sigma\in\Sym_n(\baxter)\quad \text{ if and only if }\quad  W\in\LH_n^*.$$
\end{lemma}

We do not have access to the content of \cite{Vi81} and our characterization (Definition~\ref{def:prudent}) of this restricted set of sr-Laguerre histories is different from, albeit related to, that of Viennot, so we have decided to include a proof of Lemma~\ref{lem:Viennot} here for the sake of completeness. To that end, we need a definition and an alternative way of describing Baxter permutations.
\begin{Def}
Given four indices $1\le i<j<k<\ell\le n$, we say the quadruple $(\sigma_i,\sigma_j,\sigma_k,\sigma_{\ell})$ forms a $\overline{2}41\overline{3}$ pattern in the permutation $\sigma\in\Sym_n$, if $\sigma_k<\sigma_i<\sigma_{\ell}=\sigma_i+1<\sigma_j$. In other words, the overlined letters in the pattern (in this case the $2$ and $3$) are required to be adjacent in their {\bf values} rather than positions when they are placed back to the permutation $\sigma$. Similarly, the quadruple $(\sigma_i,\sigma_j,\sigma_k,\sigma_{\ell})$ is said to form a $\overline{3}14\overline{2}$ pattern in $\sigma$, if $\sigma_j<\sigma_{\ell}<\sigma_i=\sigma_{\ell}+1<\sigma_k$. Let $\Sym_n(\dualbax)$ denote the set of permutations in $\Sym_n$ that avoid simultaneouly both patterns $\overline{2}41\overline{3}$ and $\overline{3}14\overline{2}$.
\end{Def}
\begin{proposition}
For every $n\ge 1$, we have
\begin{align}\label{dual baxter}
\Sym_{n}(\baxter) = \Sym_n(\dualbax).
\end{align}
\end{proposition}
\begin{proof}
Firstly, it is known that the set of Baxter permutations is closed under the inverse mapping $\iota:\pi\mapsto\pi^{-1}$; see for instance \cite[Lemma 2.1]{LL23}. In addition, we observe that a permutation $\pi$ contains the pattern $3\underline{14}2$ (resp.~$2\underline{41}3$) if and only if $\pi^{-1}$ contains the pattern $\overline{2}41\overline{3}$ (resp.~$\overline{3}14\overline{2}$). Combining these facts gives rise to \eqref{dual baxter}. We trust the reader to verify the details.
\end{proof}

\begin{proof}[Proof of Lemma~\ref{lem:Viennot}]
Let $\sigma\in\Sym_n$ be a permutation and $W=\Phi_{\FV}(\sigma)$ be its image in $\LH_n$. For each $1\le j<n$, we examine the way that the two consecutive letters $j$ and $j+1$ are inserted into $\sigma^{(j-1)}$ and $\sigma^{(j)}$. A careful analysis of the correspondence $\Phi_{\FV}$ leads to the following observations.
\begin{itemize}
  \item Suppose $j+1$ is inserted to the right of $j$. There exist letters $i<j$ and $k>j+1$ such that the quadruple $(j,k,i,j+1)$ forms a $\overline{2}41\overline{3}$ pattern in $\sigma$, if and only if
  \begin{align*}
    c_{j+1}-c_j\le
    \begin{cases}
    -1 & \text{ for $w_j=\ULr$},\\
    -2 & \text{ for $w_j=\DLb$}.
    \end{cases}
  \end{align*}
  See Figure~\ref{fig:insert} below for an illustration of the case $w_j=\rU$, where the label below an open slot \varspace~is exactly the value assigned to $c_j$ (resp.~$c_{j+1}$) if the letter $j$ (resp.~$j+1$) is inserted there. It is evident from the picture that $c_{j+1}-c_j\le -1$ in this case. Other cases could be verified in a similar fashion, the details are omitted.
  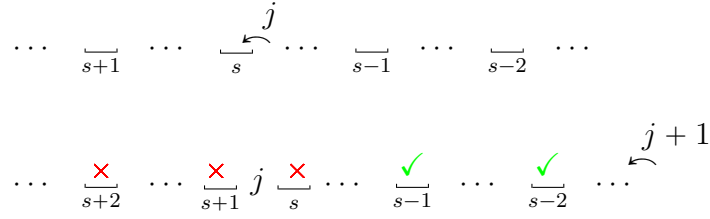
\begin{figure}[!ht]
  \begin{center}
    \begin{tikzpicture}[scale=0.9]
      \node at (0,2) {$\cdots$};
      \node at (1,1.8) {$\underset{s+1}{\mathtt{\varspace[1.2em]}}$};
      \node at (2,2) {$\cdots$};
      \node at (3,1.8) {$\underset{s}{\mathtt{\varspace[1.2em]}}$};
      \node at (4,2) {$\cdots$};
      \node at (5,1.8) {$\underset{s-1}{\mathtt{\varspace[1.2em]}}$};
      \node at (6,2) {$\cdots$};
      \node at (7,1.8) {$\underset{s-2}{\mathtt{\varspace[1.2em]}}$};
      \node at (8,2) {$\cdots$};
      \node at (3.3,2.1) {$\curvearrowleft$};
      \node at (3.5,2.5) {$j$};
      \node at (0,0) {$\cdots$};
      \node at (1,-.2) {$\underset{s+2}{\mathtt{\varspace[1.2em]}}$};
      \node at (2,0) {$\cdots$};
      \node at (3.3,-.1) {$\underset{s+1}{\mathtt{\varspace[1.2em]}}~j~\underset{s}{\mathtt{\varspace[1.2em]}}$};
      \node at (4.6,0) {$\cdots$};
      \node at (5.6,-.2) {$\underset{s-1}{\mathtt{\varspace[1.2em]}}$};
      \node at (6.6,0) {$\cdots$};
      \node at (7.6,-.2) {$\underset{s-2}{\mathtt{\varspace[1.2em]}}$};
      \node at (8.6,0) {$\cdots$};
      \node at (9,0.3) {$\curvearrowleft$};
      \node at (9.5,0.7) {$j+1$};
      \draw[color=red] plot[only marks,mark=x,mark size=4pt] (1,.2);
      \draw[color=red] plot[only marks,mark=x,mark size=4pt] (2.7,.2);
      \draw[color=red] plot[only marks,mark=x,mark size=4pt] (3.9,.2);
      \node[color=green] at (5.6,.3) {$\checkmark$};
      \node[color=green] at (7.6,.3) {$\checkmark$};
    \end{tikzpicture}
    \caption{A scenario that a $\overline{2}14\overline{3}$ pattern is formed when $j$ and $j+1$ are inserted}
    \label{fig:insert}
  \end{center}
\end{figure}

  \item Suppose $j+1$ is inserted to the left of $j$. There exist letters $i<j$ and $k>j+1$ such that the quadruple $(j+1,i,k,j)$ forms a $\overline{3}14\overline{2}$ pattern in $\sigma$, if and only if
  \begin{align*}
    c_{j+1}-c_j\ge
    \begin{cases}
    2 & \text{ for $w_j=\ULr$},\\
    1 & \text{ for $w_j=\DLb$}.
    \end{cases}
  \end{align*}
\end{itemize}
A direct consequence of the above observations is that $\sigma$ avoids both patterns $\overline{2}41\overline{3}$ and $\overline{3}14\overline{2}$ if and only if $W$ is prudent. This completes the proof in view of \eqref{dual baxter}.
\end{proof}

It is worth noting that the history $W$ in Figure \ref{fig:history} is prudent, this is in agreement with Lemma~\ref{lem:Viennot} above since its preimage $\Phi_{\FV}^{-1}(W)=5~3~2~4~7~6~1~8~10~9$ is a Baxter permutation. This characterization of Baxter permutations enabled Viennot to construct a bijection between Baxter permutations and triples of non-intersecting lattice paths. A fundamentally different bijection between these two objects was conjectured by Dilks~\cite[Conjecture 3.5]{Dil15} in his Ph.D. thesis and it was confirmed recently by Lin and Liu~\cite{LL23}.

\begin{lemma}\label{lem:xi closed}
For every $n\ge 1$, the involution $\xi: \LH_n\to \LH_n$ is closed over $\LH^*_n$.
\end{lemma}
\begin{proof}
Since $\xi$ is an involution, it suffices to show that if a history $W=(w,h,c)\in\LH_n$ is prudent, then its image $V:=\xi(W)=(v,g,b)$ is prudent as well. To that end, we need to compute the weight difference $b_{j+1}-b_j$ for every $1\le j<n$, and verify that it meets the criteria given in Definition~\ref{def:prudent}. The whole process is a simple but tedious case-by-case verification that references Table~\ref{xsi-cases}. We elaborate on two such cases and leave the remaining cases to the reader.

First note that cases 13 and 14 are irrelevant since we require that $j<n$. Let us consider cases 2 and 12. Applying the condition (4) in Definition~\ref{def:xsi} for the indices $j$ and $j+1$, we deduce that
\begin{align}\label{eq:b-c}
  b_{j+1}-b_j=g_{j+1}-g_j+h_{n+1-j}-h_{n-j}+c_{n-j}-c_{n+1-j}.
\end{align}
Next for each case that we intend to verify, we look up in Table~\ref{xsi-cases} the values of $v_j,w_{n-j},g_j$, and $g_{j+1}$ and calculate $b_{j+1}-b_j$ according to \eqref{eq:b-c}. For case 2, we see that $$(v_j,w_{n-j},g_j,g_{j+1})=(\ULr,\ULr,h_{n+1-j},h_{n-j}+1),$$ 
plugging this into \eqref{eq:b-c} yields $b_{j+1}-b_j=c_{n-j}-c_{n+1-j}+1$. Now $W$ being prudent and $w_{n-j}=\ULr$ means that $c_{n+1-j}-c_{n-j}=0$ or $1$, hence $b_{j+1}-b_j=1$ or $0$. This means the prudence condition (1) is satisfied for the step $v_j=\ULr$. For case 12, a similar calculation shows that $b_{j+1}-b_j=0$ or $-1$, and $v_j=\DLb$, agreeing with the prudence condition (2).
\end{proof}

Now we are in a position to prove our second main result.
\begin{proof}[Proof of Theorem~\ref{thm:baxter}]
Since $\widehat\phi$ ($=\phi$) is an involution on $\Sym_n(\baxter)$, it suffices to show that it maps every Baxter permutation $\sigma$ to another Baxter permutation. Indeed, we see that
$$\widehat\phi(\sigma)=\phi(\sigma)=\Phi_{\FV}^{-1}(\xi(\Phi_{\FV}(\sigma))).$$
Now Lemma~\ref{lem:Viennot} tells us that $\Phi_{\FV}(\sigma)$ is a prudent history, then Lemma~\ref{lem:xi closed} implies that $\xi(\Phi_{\FV}(\sigma))$ remains being prudent. Applying Lemma~\ref{lem:Viennot} one more time we see that $\widehat\phi(\sigma)$ is a Baxter permutation, as desired.
\end{proof}

Define  the  polynomial 
\begin{align}
Q_n(p,q,r,s):=\sum_{\sigma\in\Sym_n(\baxter)}p^{\lasc(\sigma)}q^{\ldes(\sigma)}r^{\rasc(\sigma)}s^{\rdes(\sigma)}.
\end{align}
Noting that the set of Baxter permutations $\Sym_n(\baxter)$ is closed under the maps of complementation and reversal, we have the following symmetry which parallels \eqref{bds-symmetries},
\begin{align}
Q_n(p,q,r,s)=Q_n(q,p,s,r)=Q_n(s,r,q,p)=Q_n(r,s,p,q).
\end{align}
In this context, Theorem~\ref{thm:baxter} is equivalent to the following four relations.
\begin{corollary}
We have the relations
\begin{align}
Q_n(1,q,r,s) &= Q_n(1,q,s,r),\label{eq:baxter equidist"}\\
Q_n(p,1,r,s) &= Q_n(p,1,s,r),\nonumber\\
Q_n(p,q,1,s) &= Q_n(q,p,1,s),\nonumber\\
Q_n(p,q,r,1) &= Q_n(q,p,r,1).\nonumber
\end{align}
\end{corollary}

\section{Closing remarks}
A natural question that often arises when an involutive map is introduced, is to enumerate or even characterize its set of fixed points. This is the main theme in Dilks' Ph.D. thesis~\cite{Dil15}, which investigates involutions defined on various combinatorial objects that are enumerated by the Baxter numbers. For our involution $\widehat\phi$, this question can be addressed in the following way. First we note a simple fact.
\begin{proposition}
If $f:A\to B$ is a bijection between two finite sets $A$ and $B$, while $\zeta$ and $\eta$ are involutions on $A$ and $B$, respectively, and they are linked as 
$$\zeta = f^{-1}\circ\eta\circ f,$$
then an element $a\in A$ is fixed by $\zeta$ if and only if its image $f(a)\in B$ is fixed by $\eta$. In particular, the set of fixed points of $\zeta$ is equinumerous with the set of fixed points of $\eta$.
\end{proposition}

Recall that $\widehat\phi=\Phi_{\SZ}^{-1}\circ\widehat\theta\circ\Phi_{\SZ}$, and notice that the permutations fixed by $\widehat\theta$ exist only when $n=2m+1$ is odd. In this case $\sigma\in\Sym_{2m+1}$ satisfies $\widehat\theta(\sigma)=\sigma$ if and only if it has the form
$$\sigma= \sigma_1\cdots \sigma_m (2m+1) (2m+1-\sigma_m)\cdots(2m+1-\sigma_1).$$
This observation immediately gives us the following enumerative result.

\begin{theorem}
The number of fixed points of $\widehat\phi$ is $0$ when $n$ is even, and it is $2^m m!=(2m)!!$ when $n=2m+1$ is odd.
\end{theorem}

Setting $s=1$ in \eqref{eq:baxter equidist"} produces the following identity:
\begin{align}\label{baxter analog}
Q_n(q,r):=Q_n(1,q,r,1)=\sum_{\sigma\in\Sym_n(\baxter)}q^{\ldes(\sigma)}r^{\rasc(\sigma)} = \sum_{\sigma\in\Sym_n(\baxter)}q^{\ldes(\sigma)}r^{\rdes(\sigma)},
\end{align}
which can be viewed as a Baxter permutation analogue of the two ends of \eqref{continued fraction}. Although the generating function of the Baxter numbers $\Bax_n$ does not seem to have a meaningful Jacobi type continued fraction expansion, it might still be worthwhile to find an algebraic proof of \eqref{baxter analog}. For the reader's interest, we include below the bivariate polynomials $Q_n(q,r)$ for $n=3,4,5,6,7$.
\begin{align*}
    Q_3(q,r) &= (4+q) + r,\\
    Q_4(q,r) &= (8+4q+2q^2) + (4+2q)r + 2r^2,\\
    Q_5(q,r) &= (16+12q+9q^2+4q^3+q^4) + (12+10q+5q^2+q^3)r\\
             &\quad + (9+5q+2q^2)r^2 + (4+q)r^3 + r^4,\\
    Q_6(q,r) &= (32+32q+30q^2+20q^3+12q^4+4q^5+2q^6)\\
             &\quad + (32+36q+28q^2+16q^3+6q^4+2q^5)r + (30+28q+22q^2+8q^3+4q^4)r^2 \\
             &\quad + (20+16q+8q^2+4q^3)r^3 + (12+6q+4q^2)r^4+(4+2q)r^5+2r^6,\\
    Q_7(q,r) &= (64+80q+88q^2+73q^3+56q^4+34q^5+20q^6+9q^7+4q^8+q^9)\\ 
             &\quad + (80+112q+111q^2+86q^3+56q^4+30q^5+14q^6+5q^7+q^8)r\\
             &\quad + (88+111q+112q^2+75q^3+47q^4+21q^5+9q^6+2q^7)r^2\\
             &\quad + (73+86q+75q^2+48q^3+25q^4+10q^5+3q^6)r^3\\
             &\quad + (56+56q+47q^2+25q^3+12q^4+3q^5)r^4 + (34+30q+21q^2+10q^3+3q^4)r^5\\
             &\quad + (20+14q+9q^2+3q^3)r^6 + (9+5q+2q^2)r^7+(4+q)r^8+r^9.
\end{align*}

With $Q_n(q,r)$ expressed in this way as a polynomial in $(\Z[q])[r]$, we make some observations on its specializations. Firstly, we note that $Q_n(q,0)$ is a polynomial refining the Catalan numbers $C_n=\frac{1}{n+1}\binom{2n}{n}$, and it agrees with the $t=1$ evaluation of the $(q,t)$-Catalan numbers previously studied in \cite[Theorem 1.1 choice \# 7]{FTHZ19}. On the other hand, the coefficient of $r$ in $Q_n(q,r)$ appears to be related to the so-called ``Touchard distribution'' \cite[A091894]{OEIS}. We shall explore this connection elsewhere.

Finally, we remark that it is somewhat unexpected that a single involution $\widehat\phi$ could prove two trivariate symmetries given in Theorems~\ref{thm:hatphi} and \ref{thm:baxter}. It might be worthwhile to look for other subclasses of permutations that are closed under $\widehat\phi$.

\section*{Acknowledgement}
This work was partially supported by the National Natural Science Foundation of China grants (No.~12171059 and No.~12271301).

\end{document}